\documentclass[a4paper,11pt,leqno]{amsart}
\usepackage{cite}
\usepackage[english]{babel}
\usepackage[utf8]{inputenc}
\usepackage[T1]{fontenc}
\usepackage{amsthm}
\usepackage{amssymb}
\usepackage{amsmath}

\usepackage[usenames]{color}

\usepackage{amsfonts}
\usepackage{graphicx}
\usepackage[all]{xy}
\newtheorem{Thm}{Theorem}
\newtheorem{Prop}[Thm]{Proposition}
\newtheorem{Cor}[Thm]{Corollary}

\theoremstyle{definition}

\newtheorem{Not}[Thm]{Remark}

\renewcommand{\r}{|}
\newcommand{\R}{\mathbb{R}}

\newcommand{\Z}{\mathbb{Z}}

\newcommand{\id}{\operatorname{id}}

\newcommand{\ra}{\rightarrow}

\newcommand{\act}{\curvearrowright}
\newcommand{\Deck}{\mathrm{Deck}}
\newcommand{\interior}{\mathrm{int}}

\newcommand{\Fix}{\mathcal{G}}

\newcommand{\Sym}{\mathrm{Sym}}

\begin{document}

\title[Local monodromy of branched covers]{Local monodromy of branched covers and dimension of the branch set}

\author{Martina Aaltonen}
\address{Department of Mathematics and Statistics, P.O. Box 68 (Gustaf H\"allstr\"omin katu 2b), FI-00014 University of Helsinki, Finland}
\email{martina.aaltonen@helsinki.fi}

\author{Pekka Pankka}
\address{Department of Mathematics and Statistics, P.O. Box 35, FI-40014 University of Jyv\"askyl\"a, Finland}
\email{pekka.pankka@jyu.fi}

\thanks{M.A. has been supported by the Emil Aaltonen foundation and P.P. partly supported by the Academy of Finland project 256228.}  


\date{\today}

\begin{abstract}
We show that, if the local dimension of the branch set of a discrete and open mapping $f\colon M\to N$ between $n$-manifolds is less than $(n-2)$ at a point $y$ of the image of the branch set $fB_f$, then the local monodromy of $f$ at $y$ is perfect.
In particular, for generalized branched covers between $n$-manifolds the dimension of $fB_f$ is exactly $(n-2)$ at the points of abelian local monodromy. As an application, we show that a generalized branched covering $f\colon M \to N$ of local multiplicity at most three between $n$-manifolds is either a covering or $fB_f$ has local dimension $(n-2)$. 
\end{abstract}

\maketitle

\section{Introduction}

A continuous mapping $f\colon X\to Y$ between topological spaces is a \emph{(generalized) branched cover} if $f$ is discrete and open, that is, pre-image $f^{-1}(y)$ of a point $y\in Y$ is a discrete set and $f$ maps open sets to open sets. The name branched cover for these maps stems from the Chernavskii--V\"ais\"al\"a theorem \cite{Chernavski, Vaisala}: \emph{the branch set of a branched cover between (generalized) manifolds has codimension at least two}. It is an easy consequence of the Chernavskii--V\"ais\"al\"a theorem that branched covers between (generalized) manifolds are, at least locally, completions of covering maps.

We follow here the typical naming convention in this context and say that a point $x\in X$ is a \emph{branch point of $f$} if $f$ is not a local homeomorphism at $x$. The \emph{branch set of the mapping $f$}, i.e.\;the set of branch points of $f$, is denoted $B_f$. Note that, in the context of PL topology, $B_f$ is called the singular set and its image $fB_f$ the branch set.

Branch sets of branched covers between surfaces are well-understood. By the classical Stoilow theorem (see e.g.\;\cite{Whyburn-book}), the branch set of a branched cover between surfaces is a discrete set. In higher dimensions, branch sets of PL branched covers between manifolds are subcomplexes of codimension at least two.
More general branched covers may, however, exhibit also wilder branching behavior. Heinonen and Rickman constructed in \cite{Heinonen-Rickman-Topology} and \cite{Heinonen-Rickman-Duke} quasiregular, even BLD, branched covers $S^3\to S^3$ which contain wild Cantor sets in their branch sets; see also \cite{Pankka-Rajala-Wu}. In fact, in dimensions $n\ge 3$, branch sets of branched covers are not understood in a similar precise fashion as in two dimensions. In particular a conjecture of Church and Hemmingsen \cite{Church-Hemmingsen} is still open: \emph{The branch set of a branched cover between $3$-manifolds has topological dimension one}. 

It is easy to observe that the conjecture of Church and Hemmingsen is equivalent to the question whether there exists a branched cover between $3$-manifolds for which $fB_f$ is a wild Cantor set in a neighborhood of a point in $fB_f$; see also Church \cite{Church} and Montesinos \cite{Montesinos} for related questions. Note that, we have $\dim B_f = \dim f^{-1}B_f = \dim fB_f$ for branched covers $f\colon M\to N$ between manifolds by \cite[Corollary 2.3]{Church-Hemmingsen}.

In this article we consider the connection of the monodromy to the local dimension of the branch set. This question is interesting already in the context of PL branched covers as the following example shows.

Let $f \colon S^3 \to H^3$ be a normal covering of the Poincar\'e homology sphere $H^3$ and $F=\Sigma^2 f \colon \Sigma^2 S^3\to \Sigma^2 H^3$ the double suspension of $f$. Then $F$ is a normal branched cover for which $B_F$ is a circle in $S^5=\Sigma^2 S^3$ and $FB_F$ in $\Sigma^2 H^3 \cong S^5$ is a wild knot; see Edwards \cite{Edwards} and Cannon \cite{Cannon}. In particular, $B_F$ and $FB_F$ both have codimension $4$. The monodromy group $G_F$ of $F$ is isomorphic to the fundamental group of $H^3$ which is a perfect group. Recall, that a group $\Gamma$ is \emph{perfect} if $\Gamma/[\Gamma,\Gamma]$ is a trivial group. Our main result shows that this is a general phenomenon: \emph{If the branch set of a branched cover has codimension larger than two, then the local monodromy groups of the map are perfect}.

We define the local monodromy of a branched cover as follows. Let $f\colon M\to N$ be a proper branched cover between $n$-manifolds. By a result of Berstein and Edmonds \cite{Berstein-Edmonds} (see also \cite{Aaltonen}), there exists a space $X_f$ and an action $G_f \act X_f$ of the monodromy group $G_f$ of $f$ by homeomorphisms for which the diagram
\begin{equation}
\label{eq:triangle}
\xymatrix{
& X_f \ar@(r,u)_{G_f} \ar[dl] \ar[dr]^{\bar f}& \\
M \ar[rr]^f & & N }
\end{equation}
commutes, where the maps $X_f \to M$ and $\bar f \colon X_f \to N$ are normal branched covers. Recall that, similarly as for covering maps, a branched cover $h \colon X\to Y$ is \emph{normal} if $h$ is a quotient map of the action of the deck group $\Deck(h)$ to $X$.

We call the map $\bar f\colon X_f \to N$, which is the orbit map of the action $G_f \act X_f$, the \emph{normalization of $f$}; in particular, $\Deck(\bar f) = G_f$. Given $y\in N$, the stabilizer subgroups of $G_f$ of points in $\bar f^{-1}(y)$ are conjugate to each other and we define the \emph{local monodromy $\Fix_f(y)$ of $f$ at $y\in N$} to be the conjugacy class of these subgroups.

Our main theorem reads as follows.
\begin{Thm}
\label{thm:CH_perfect}
Let $f\colon M\to N$ be a proper branched cover between $n$-manifolds. If the local dimension of $fB_f$ at $y\in fB_f$ is less than $n-2$, then $\Fix_f(y)$ is a finite perfect group. 
\end{Thm}

As an immediate corollary of Theorem \ref{thm:CH_perfect} and the Chernavskii--V\"ais\"al\"a theorem on the dimension of the branch set, we obtain an elementary proof for the following well-known result; for a proof using the Smith theory, see e.g.\;\cite{Borel-book}.

\begin{Cor}
Let $f\colon M\to N$ be a proper normal branched cover between $n$-manifolds having abelian deck group. Then either $f$ is a covering map or $\dim B_f = \dim fB_f = n-2$. 
\end{Cor}

As an application of Theorem \ref{thm:CH_perfect}, we also obtain a positive result in a special case of the conjecture of Church and Hemmingsen for branched covers having local multiplicity at most three. More precisely we have the following result.

\begin{Thm}
\label{thm:CHA}
Let $f\colon M \to N$ be a proper branched cover between $n$-manifolds so that the local multiplicity of $f$ is at most three in $B_f$. Then either $f$ is a covering map or $fB_f$ has local dimension $n-2$. 
\end{Thm}

A remark on the relation of these results to the classical Smith theory is in order. For normal branched covers $f\colon M\to N$ between (cohomology) manifolds, the branch set $B_f$ has decomposition into finitely many cohomology manifolds of codimension at least $2$ (see e.g.\;\cite[Theorem V.2.2]{Borel-book}). Since $B_f$ is not a (cohomology) $0$-manifold, it is therefore easy to conclude that $fB_f$ has local dimension at least $1$ at each point. It is not known to us to which extent these methods are available in the context of Theorem \ref{thm:CH_perfect}, since the space $X_f$ is not \emph{a priori} a cohomology $n$-manifold.

We finish this introduction with a non-existence result for branched covers branching over an Antoine's necklace. A branched cover $f \colon X\to Y$ is \emph{locally normal} if each point $x\in X$ has a neighborhood $U$ for which $f|_U \colon U \to fU$ is a normal branched cover.


\begin{Thm}
\label{thm:Antoine}
There are no locally normal branched covers $S^3\to S^3$ for which the image of the branch set is an Antoine's necklace.
\end{Thm}

\medskip

This article is organized as follows. After discussing preliminaries in Section \ref{sec:pre}, we prove in Section \ref{sec:CH_perfect} a slightly more general version of Theorem \ref{thm:CH_perfect} for branched covers from codimension $2$ manifold completions to manifolds. In Section \ref{sec:CHA} we give a proof of a similar generalization for Theorem \ref{thm:CHA}. Finally, in Section \ref{sec:applications}, we discuss applications to Cantor sets and prove Theorem \ref{thm:Antoine}.

\vspace*{5mm}
\noindent{\bf Acknowledgments}

Authors thank Juan Souto and Jang-Mei Wu for discussions and their comments on the manuscript.


\section{Preliminaries}
\label{sec:pre}

In this section, we recall few basic facts on branched covers and introduce some terminology which we use in the forthcoming sections. Note that we consider only path-connected manifolds and manifold completions. 

\subsection{Coverings}

Let $f \colon X \to Y$ be a covering map between path-connected spaces and $y_0\in Y$. The \emph{monodromy $\mu_f \colon \pi_1(Y,y_0) \to \Sym f^{-1}(y_0)$ of $f$} is the homomorphism which associates  a permutation of $\Sym f^{-1}(y_0)$ to every homotopy class $[\gamma]$. More precisely, let $\gamma \colon [0,1]\to Y$ be a loop based at $y_0$, $x\in f^{-1}(y_0)$, and $\tilde \gamma_x \colon [0,1]\to X$ the lift of $\gamma$ from $x$ in $f$. Then $\mu_f([\gamma])(x) = \tilde \gamma_x(1)$. The \emph{monodromy group $G_f$ of $f$} is the quotient $\pi_1(Y,y_0)/\ker \mu_{f}$. 

For normal coverings $f\colon (X,x_0)\to (Y,y_0)$ of pointed spaces, we have also the \emph{deck homomorphism $\sigma_{f,x_0} \colon \pi_1(Y,y_0) \to \Deck(f)$ of $f$} which associates to each homotopy class $[\gamma]\in \pi_1(Y,y_0)$ a deck transformation using the lift of the representative $\gamma$ from $x_0$, that is, given a loop $\gamma \colon [0,1] \to Y$ at $y_0$, we set $\sigma_{f,x_0}([\gamma])$ to be the (unique) deck transformation $\tau \colon X\to X$ satisfying $\tau(x) = \tilde \gamma_{x_0}(1)$. Note that, $G_f \cong \Deck(f)$, but typically, the deck transformation $\sigma_{f,x_0}([\gamma])$ is not an extension of the permutation $\mu_f([\gamma])$.

\subsection{Manifold completions and proper branched coverings}
\label{sec:BC}

We say that a locally connected (and locally compact) Hausdorff space $X$ is a \emph{codimension $2$ manifold completion}, if there exists a connected $n$-manifold $X^o$ (possibly with boundary) and an embedding $\iota \colon X^o \hookrightarrow X,$ so that $\iota(X^o) \subset X$ is dense and the set $X\setminus \iota(X^o)$ does not locally separate $X$. In particular, $\overline{X^o} = X$. In other words, we have obtained $X$ from $X^o$ by the Fox-completion \cite{Fox}; see also \cite{Montesinos}. 


This class of spaces rises naturally in the context of branched covers. Indeed, if $f\colon X\to M$ is a branched cover from a locally compact and locally connected Hausdorff space $X$ to an $n$-manifold $M$ so that $B_f$ does not locally separate $X$. Then $X$ is a codimension $2$ manifold completion, since $X^o = X\setminus B_f$ is an $n$-manifold. 

Let $f\colon X\to M$ be a proper branched cover from a codimension $2$ manifold completion to a manifold $M$. Then $f$ is a completed cover, that is, $f$ is the unique extension of the covering $f' = f|_{X'} \colon X' \to M'$ with respect to $M'$, where $X' = X\setminus f^{-1}fB_f$ and $M' = M\setminus fB_f$ are open dense subsets of $X$ and $M$, respectively, and the sets $f^{-1}fB_f$ and $fB_f$ do not locally separate $X$ and $M$, respectively. Indeed, since $f$ is proper, $f$ is surjective and both $fB_f$ and $f^{-1}fB_f$ are closed sets. For the general theory of these completions, see e.g.\;Fox \cite{Fox}, Berstein-Edmonds \cite{Berstein-Edmonds}, Edmonds \cite{Edmonds-Michigan}, or \cite{Aaltonen}. We call $f'$ the \emph{regular part of $f$}. 

We recall two facts on proper branched covers. First, $f\colon X\to M$ is a proper normal branched cover if and only if its regular part $f'\colon X'\to M'$ is a proper normal covering; see Edmonds \cite{Edmonds-Michigan}. We also recall that the homomorphism $\Deck(f) \to \Deck(f')$, $g\mapsto g|_{X'}$, is an isomorphism; see Montesinos \cite{Montesinos-old}. 

We define the \emph{monodromy of $f \colon X\to M$} to be the monodromy of its regular part $f' \colon X' \to M'$, that is, 
\[
\mu_{f} := \mu_{f'} \colon \pi_1(M',y_0) \to \Sym f^{-1}(y_0)
\]
for $y_0 \in M'$.

\subsection{Monodromy triangle}

The regular part $f' \colon X' \ra M'$ of a branched covering $f \colon X \ra M$ has, by the classical covering space theory, a monodromy triangle 
\begin{equation*}
\xymatrix{
& \widetilde{M'}/\ker \mu_{f'} \ar[dl]_{p} \ar[dr]^{\bar f'}& \\
X' \ar[rr]^{f'} & & M'}
\end{equation*}
where the normal covering map $\bar f'$ is the orbit map of the natural action of the monodromy group $G_{f'} = \pi_1(M',y_0)/\ker \mu_{f'}$ on $\widetilde M'/\ker \mu_{f'}$ by homeomorphisms. The monodromy triangle \eqref{eq:triangle} of $f \colon X \ra M$ is obtained as an extension of the monodromy triangle of its regular part $f' \colon X' \ra M'.$ We refer to Berstein-Edmonds \cite{Berstein-Edmonds}, or \cite{Aaltonen}, for details.

Note that, given a normalization $\bar f\colon X_f \to M$ of a proper branched cover $f\colon X\to M$ and a subgroup $H \subset G_f$ of the monodromy group $G_f$ of $f$, there exists a factorization
\[
\xymatrix{
X_f \ar[d]_q \ar[dr]^{\bar f} & \\
X_f/H \ar@{-->}[r]_{\bar f_H} & M }
\]
where $q$ and $\bar f_H$ are branched covers induced by the action $H\act X_f$. Moreover, if $H$ is normal in $G_f$, then $\bar f_H$ is a normal branched cover with $\Deck(\bar f_H) \cong \Deck(\bar f)/H$. 

\subsection{Regular neighborhoods}

For the localization arguments we recall the existence of normal neighborhoods (V\"ais\"al\"a \cite[Lemma 5.1]{Vaisala}): \emph{Let $f\colon X\to M$ be a branched cover from a codimension $2$ manifold completion to a manifold. Then, for each $x\in X$, there exists a neighborhood $U\subset X$ of $x$ for which $f|_U \colon U \to fU$ is a proper branched cover. Moreover, for each domain $V$ compactly contained in $fU$ and each component $W\subset f^{-1}V\cap U$, the restriction $f|_W \colon W\to V$ is a proper branched cover.}


\subsection{Topological dimension and cohomology}

In what follows, we call the covering dimension of a space simply as dimension. Recall that a \emph{space $X$ has covering dimension at most $n$} (denoted $\dim X \le n$) if each covering of $X$ has a refinement of local multiplicity at most $n+1$. Further, \emph{$X$ has dimension $n$} ($\dim X = n$) if $\dim X \le n$ and $X$ does not have covering dimension at most $n-1$. 
See e.g.\;Engelking \cite{Engelking-book} for comparisons with other definitions.

A closed set $E\subset X$ has \emph{local dimension at most $n$ at $x\in E$} if there exists a neighborhood $U\subset X$ of $x$ so that $\dim(U\cap E)\le n$. Similarly, $E$ has \emph{local dimension at least $n$} if for all neighborhoods $U\subset X$ of $x$ for which $\dim (U\cap E) \ge n$.

In the proof of Theorem \ref{thm:CH} we use the fact that the Alexander--Spanier (or equivalently \v{C}ech-cohomology) groups $H^k(X;\Z)$ are trivial for $k>n$ if $\dim X\le n$; see e.g.\;\cite[pp.94-95]{Engelking-book}. Note that a codimension $2$ manifold completion is a Cantor manifold in the sense of \cite[Definition 1.9.5]{Engelking-book}. 



\section{Proof of Theorem \ref{thm:CH_perfect}}
\label{sec:CH_perfect}

In this section we prove the following version of Theorem \ref{thm:CH_perfect}.

\begin{Thm}
\label{thm:CH_perfect_B}
Let $X$ be a codimension $2$ manifold completion, $M$ an $n$-manifold, $f\colon X\to M$ a proper branched cover and $y \in fB_f.$  If the local dimension of $fB_f$ at $y\in fB_f$ is less than $n-2$, then $\Fix_f(y)$ is a finite perfect group. 
\end{Thm}

We begin with the abelian case of the theorem.
\begin{Thm}
\label{thm:CH}
Let $X$ be a codimension $2$ manifold completion, $M$ an $n$-manifold, and $f\colon X\to M$ a proper branched cover for which $B_f\ne \emptyset$. The local dimension of $fB_f$ is $(n-2)$ at points of abelian local monodromy. 
\end{Thm}

By Chernavskii-V\"ais\"al\"a theorem, the local dimension of $fB_f$ is at most $n-2$. Thus we may assume in this section that $n\ge 3$ and show that the dimension of $fB_f$ is at least $n-2$ at points of abelian local monodromy.  

\begin{Not} 
There are simple examples of branched covers having points where the local monodromy is abelian and not cyclic. For example, let $f\colon \R^3 \to \R^3$ be the composition $f=f_1 \circ f_2$, where $f_1 \colon \R^3 \ra \R^3$ is a $2$-to-$1$ winding map around the $x$-axis and $f_2 \colon \R^3 \ra \R^3$ is a $2$-to-$1$ winding map around the $y$-axis. 
Then $f$ is a proper $4$-to-$1$ normal branched cover for which the local monodromy at the origin is the abelian, but non-cyclic group, $\Z_2 \times \Z_2$. At every other point of $fB_f$ the local monodromy group is the cyclic group $\Z_2$.
\end{Not}

\begin{proof}[Proof of Theorem \ref{thm:CH}]
Suppose there exists $y\in fB_f$ for which the local dimension of $fB_f$ at $y$ is less than $n-2$. Let $\bar f \colon X_f \to M$ be a normalization of $f$ and $x\in \bar f^{-1}(y)$. We show first that there exists a neighborhood $U\subset X$ of $x$, having closure $E=\overline{U}$, so that $\bar f E$ is an $n$-cell in $M$, and $\bar f|_E \colon E \to \bar f E$ is a proper branched cover and $\dim (\bar f|_EB_{\bar f|_E}) < n-2$. 


Let $W$ be a neighborhood of $y$ in $M$ contained in an $n$-cell and for which $\dim (W\cap fB_f) < n-2$. Then, by \cite[Lemma 5.1]{Vaisala}, we may fix a neighborhood $V\subset X$ of $x$ so that $fV \subset W$ and $f|_V \colon V \to fV$ is a proper map. Let $D\subset \interior (fV \setminus f(\partial V))$ be an open $n$-cell so that $\overline{D}$ is an $n$-cell. Then $f|_{\overline{U}} \colon \overline{U} \to \overline{D}$, where $U$ is the $x$-component of $f^{-1}D$, is a proper branched cover.

We construct now a double of $\bar f|_E \colon E \to \bar fE$ as follows. Let $Z$ be the quotient space obtained by gluing two copies of $E$ together along the boundary. More precisely, let $Z = \left(E\times \{1,2\}\right)/{\sim}$, where $\sim$ is the minimal equivalence relation satisfying $(x,1) \sim (x,2)$ for $x\in \partial E$. Let $q_E\colon E\times \{1,2\}\to Z$ be the quotient map $(x,i) \mapsto [(x,i)]$. The non-manifold points of $Z$ are contained in the set $q_E((B_{\bar f} \cap E)\times \{1,2\})$ and $Z$ is a codimension $2$ manifold completion. 

Let also $S = \left( (\bar fE\times \{1\}) \cup (\bar fE\times \{2\})\right)/{\sim}$ be an $n$-sphere obtained by gluing the $n$-cells $\bar fE\times \{i\}$ together along the boundary $\partial \bar fE$ similarly as above. Let $q \colon \bar fE\times \{1,2\} \to S$ be the associated quotient map.

Let $g \colon Z\to S$ be the unique map for which the diagram 
\[
\xymatrix{
E\times \{1,2\} \ar[d]_{q_U} \ar[r]^{\bar f|_E \times \id} & \bar fE\times \{1,2\} \ar[d]^{q} \\
Z \ar[r]_{g} & S }
\]
commutes. Then $g$ is an open and discrete map. Indeed, discreteness of $g$ follows immediately. For the openness of $g$, it suffices to observe that, given a neighborhood $V$ of a point in $q_E(\partial E \times \{1\})$, there exists open sets $V_1$ and $V_2$ in $E$ so that $q_E( (V_1 \times \{1\}) \cup (V_2 \times \{2\}) ) = V$. Then 
\[
gV = q( (\bar fV_1 \times \{1\}) \cup (\bar fV_2 \times \{2\}))
\]
is an open set in $S$. Thus $g$ is a branched cover. Similarly we observe that $g$ is, in fact, a normal branched cover having an abelian deck group $\Deck(g) \cong \Deck(\bar f|_E)$ and $g B_{g} \subset q( \bar f|_E B_{\bar f|_E} \times \{1,2\})$.

Since $\dim(g B_{g}) < n-2$ and $n\ge 3$, we have, by the Alexander duality (see e.g.\;\cite[Theorem 6.6]{Massey-book}), that  
\[
H_1(S^n\setminus gB_{g}) \cong H^{n-2}(gB_{g};\Z) = 0.
\]
Thus $\pi_1(S^n\setminus gB_{g},z_0)$ is a perfect group for every $z_0 \in S^n\setminus gB_{g}.$

Let  
\[
g':=g|Z \setminus g^{-1}(gB_g)\colon Z\setminus g^{-1}(gB_g) \to S^n \setminus gB_g
\]
be a restriction of $g$ and let $\varphi_{g',y_0} \colon \pi_1(S^n\setminus gB_{g},z_0) \to \mathrm{Deck}(g')$ be the deck-homomorphism for points $z_0 \in S^n\setminus gB_{g}$ and $y_0 \in g'^{-1}\{z_0\}.$ Then $\mathrm{Deck}(g') \cong \mathrm{Deck}(g)$ is abelian. Thus $\mathrm{Deck}(g')$ is an abelian image of a perfect group, and hence trivial. We conclude that then $\mathrm{Deck}(g)$ is also trivial and the normal branched cover $g$ is a homeomorphism. Hence also $\bar f|_E$ is a homeomorphism and $\Fix_{f}(y)$ is trivial. This is a contradiction, since $y\in fB_f \cap U$. Hence the local dimension of $fB_f$ at each of its points is at least $n-2$. 
\end{proof}

\begin{proof}[Proof of Theorem \ref{thm:CH_perfect_B}] 
Suppose the local dimension of $fB_f$ at $y\in fB_f$ is less than $n-2$. Let $G_f$ be the monodromy group of $f$ and let
\[
\xymatrix{
& X_f \ar@(r,u)_{G_f} \ar[dl]_q \ar[dr]^{\bar f} & \\
X \ar[rr]^f & & M }
\]
be the monodromy triangle of $f$, where $\bar f$ is the normalization of $f$. We need to show that $\Fix_{\overline{f}}(z)=[\Fix_{\overline{f}}(z),\Fix_{\overline{f}}(z)]$ for a point $z \in \bar f ^{-1}(y).$

Let $V_0$ be such a neighborhood of $y$ that the dimension of $V_0 \cap fB_f$ is less than $n-2$ and let $z\in \bar f^{-1}(y)$. By V\"ais\"al\"a's lemma \cite[Lemma 5.1]{Vaisala}, we may fix a neighborhood $W$ of $z$ for which $\bar fW \subset V_0$ is simply connected, $\bar f^{-1}(\bar f(z)) = \{z\}$, and that the restriction $\bar f|_W \colon W\to \bar fW$ is a proper branched cover. We denote $U=qW$ and $V=\bar fW$. Then $g:=f|_U \colon U \to V$ is a proper branched cover and we have the diagram
\[
\xymatrix{
& W \ar[dl]_{q|_W} \ar[dr]^{\bar f|_W}& \\
U \ar[rr]^g & & V}
\]
where $q \r W$ and $\bar f \r W$ are normal branched coverings and $\Fix_{\bar f}(z) \cong \Deck(\bar f|_W).$

Let $h = \bar f|_W \colon W\to V$, and denote $N = [\Deck(h),\Deck(h)]$. We factor the normal branched cover $h$ as
\[
\xymatrix{
W \ar[dr]^{h} \ar[d]_p & \\
W/N \ar@{-->}[r]_{h_N} & V}
\]
where $p\colon W \to W/N$ is the quotient map $x\mapsto Nx$. Since $N\subset \Deck(h)$ is normal, $h_N \colon W/N \ra V$ is a normal branched covering and $\Deck(h_N) \cong \Deck(h)/N$. Hence $\Deck(h_N)$ is abelian. Since $(h_N)B_{h_N} \subset fB_f$ and the dimension of $V \cap fB_f$ is less than $n-2,$ we have by Theorem \ref{thm:CH} that $h_N$ is a covering map. Since $V$ is simply-connected, $h_N$ is a homeomorphism. Thus $\Deck(\bar f|_W) = N$.
\end{proof}


\section{Proof of Theorem \ref{thm:CHA}}
\label{sec:CHA}

We prove now a version of Theorem \ref{thm:CHA} for branched covers having a codimension $2$ manifold completion as a domain. The statement reads as follows.

\begin{Thm}
\label{thm:CHA_B}
Let $f\colon X \to M$ be a proper branched cover from a codimension $2$ manifold completion $X$ to an $n$-manifold $M$ so that the local multiplicity of $f$ is at most $3$ in $B_f$. Then either $f$ is a covering or $fB_f$ has local dimension $n-2$. 
\end{Thm}

\begin{proof}
Suppose that $f$ is not a covering map. We show that the local dimension of $fB_f$ is $n-2$ at each point of $fB_f$. By the Chernavskii--V\"ais\"al\"a theorem, the local dimension of $fB_f$ is at most $n-2$ at each point $y\in fB_f$. Thus it suffices to show that the local dimension of $fB_f$ at $f(x)$ is at least $n-2$. 

Suppose first that there exists $x\in B_f$ for which the local multiplicity of $f$ at $x$ is $2$. Then $\Fix_f(f(x)) \cong \Z_2$ and, by Theorem \ref{thm:CH}, the local dimension of $fB_f$ is $n-2$.

Suppose now that the local multiplicity of $f$ at $x\in B_f$ is $3$. Let $G_f$ be the monodromy group of $f$ and let
\[
\xymatrix{
& X_f \ar@(r,u)_{G_f} \ar[dl]_q \ar[dr]^{\bar f} & \\
X \ar[rr]^f & & M }
\]
be the monodromy triangle of $f$, where $\bar f$ is the normalization of $f$. 

Let $U_0$ be a neighborhood of $x$ so that $f|_{U_0} \colon U_0\to M$ has multiplicity at most $3$. Let now $z\in q^{-1}(x)$. As in the proof of Theorem \ref{thm:CH_perfect} we fix, using V\"ais\"al\"a's lemma \cite[Lemma 5.1]{Vaisala}, a neighborhood $W$ of $z$ so that $\bar f|_W \colon W\to \bar fW$ is a proper branched cover, $\bar f\r W ^{-1}(\bar f(z)) = \{z\}$, and $qW \subset U_0$. We denote (again) $U=qW$ and $V=\bar fW$. Then $g=f|_U \colon U \to V$ is a proper branched cover with multiplicity $3$ and we have the diagram
\[
\xymatrix{
& W \ar[dl]_{q|_W} \ar[dr]^{\bar f|_W}& \\
U \ar[rr]^g & & V}
\]
where $q \r W$ and $\bar f \r W$ are normal branched coverings.

From the monodromy triangle of $g$, we obtain the diagram
\[
\xymatrix{
& U_g \ar@(r,ur)_{G_g} \ar@(l,ul)^{G_p}\ar[dl]_p \ar[dr]^{\bar g}& \\
U \ar[rr]^g & & V}
\]
where $\bar g$ is the normalization of $g$, $G_g$ is the monodromy group of $g$, and $G_p \subset G_g$ is the monodromy group of $p$. 

Further, by minimality of the monodromy factorization \cite[Section 2.2]{Aaltonen}, there exists a branched covering $r\colon W \to U_g$ and a commutative diagram
\[
\xymatrix{
& W \ar[d]^r \ar@/_/[ddl]_{q|_W} \ar@/^/[ddr]^{\bar f|_W} & \\
& U_g \ar[dl]_p \ar[dr]^{\bar g}& \\
U \ar[rr]^g & & V}
\]

Since $g$ has multiplicity $3$, the monodromy group $G_g$ is isomorphic to a subgroup of the symmetric group $S_3$. Since $G_g$ acts transitively on $\bar g^{-1}(y_0)$, we have $|G_g|\ge 3$. Thus either $|G_g|=3$ or $|G_g|=6$, i.e.\;$G_g\cong \Z_3$ or $G_g\cong S_3$.

Suppose $G_g\cong \Z_3$. Then the local monodromy of $\bar g$ at $f(x)$ is abelian and, by Theorem \ref{thm:CH}, $\bar gB_{\bar g}$ has dimension $n-2$ at $f(x)$. Since $\bar gB_{\bar g} \subset fB_f$, we have that $fB_f$ has local dimension at $f(x)$ at least $n-2$.

Suppose now that $G_g \cong S_3$. We show that $x\in pB_p$ and that the local monodromy of $p$ at $x$ is abelian. Then, by Theorem \ref{thm:CH}, $pB_p$ has local dimension at least $n-2$ at $x$.

Since $\bar g$ has multiplicity $6$ and $g$ has multiplicity $3$, the branched cover $p$ has multiplicity $2$. Thus $G_p\cong \Z_2$. Moreover, $\Fix_p(x)\cong \Z_2$ and, in particular, the local monodromy of $p$ at $x\in pB_p$ is abelian. Indeed, since $q^{-1}(x)\cap W = \{z\}$, we have that $p^{-1}(x) = \{r(z)\}$. Thus $G_p$ fixes $r(z)$. 

Since $pB_p$ has has local dimension at least $n-2$ at $x$, we have, by the Church--Hemmingsen theorem \cite[Corollary 2.3]{Church-Hemmingsen}, that $f(pB_p)$ has local dimension at least $n-2$ at $f(x)$. Since $f(pB_p) \subset fB_f$, the proof is complete.
\end{proof}


\section{Application to toroidal Cantor sets}
\label{sec:applications}

We finish an application of Theorem \ref{thm:CH} to branched covers branching over toroidal Cantors sets. A Cantor set $C$ in a $3$-manifold $M$ is \emph{toroidal} if every (finite) covering $\{U_i\}_{i\ge 1}$ of $C$ has a (finite) refinement $\{T_j\}_{j\ge 1}$ so that each domain $T_j$ has the $2$-torus as a boundary, i.e.\;$\partial T_j \approx S^1\times S^1$, and $C\cap \partial T_j = \emptyset$ for each $j\ge 1$.

We begin with a proposition; we define that the \emph{boundary $\partial X$ of a codimension $2$ manifold completion $X$} is the boundary $\partial X^o$ of the regular part $X^o$ of $X$.

\begin{Prop}
\label{prop:CH_cor}
Let $X$ be a compact codimension $2$ manifold completion having connected boundary $\partial X$, $T$ a compact $3$-manifold with boundary $\partial T \approx S^1\times S^1$, and $f\colon X\to T$ a normal branched covering. If $f|_{\partial X} \colon \partial X\to \partial T$ is a covering, then $\dim fB_f =1$.
\end{Prop}
\begin{proof}
Clearly the restriction $f|_{\partial X} \colon \partial X\to \partial T$ is a normal covering. The homomorphism $\Deck(f) \mapsto \Deck(f|_{\partial X})$, $g \mapsto g|_{\partial X}$, is well-defined and injective. Indeed, since $f$ is open, $f^{-1}(\partial T) = \partial X$ and $g|_{\partial X}\colon \partial X\to \partial X$ is well-defined and belongs to the deck group of $f|_{\partial X}$. 

To show that the restriction $g\mapsto g|_{\partial X}$ is injective, let $g$ and $h$ be deck transformations of $f$ satisfying $g|_{\partial X} = h|_{\partial X}$. By uniqueness of deck transformations of covering maps, we obtain $g|_{X \setminus f^{-1}fB_f} = h|_{X \setminus f^{-1}fB_f}$. Thus $g=h$ by the density of $X\setminus f^{-1}fB_f$ in $X$.
 
Since $\partial X$ is connected and $f|_{\partial X}\colon \partial X\to \partial T$ is a covering, we conclude that $\partial X$ is a $2$-torus and $f|_{\partial X}$ is a normal covering. In particular, $\Deck(f|_{\partial X})$ is abelian. Thus $\Deck(f)$ is abelian and the claim follows from Theorem \ref{thm:CH}.
\end{proof}

As a consequence we obtain that branched covers do not branch over toroidal Cantor sets. We formulate this as follows. Theorem \ref{thm:Antoine} for Antoine's necklaces is a particular case of this corollary.

\begin{Cor}
\label{cor:CH2}
Let $M$ and $N$ be $3$-manifolds and $f\colon M\to N$ a locally normal branched cover so that $fB_f$ is contained in a toroidal Cantor set. Then $f$ is a covering map. 
\end{Cor}
\begin{proof}
Let $y \in fB_f$, $x\in f^{-1}(y)$, and $T\subset N$ a neighborhood of $y$ so that $\partial T$ is a $2$-torus in $N \setminus fB_f$, $f^{-1}(N\setminus T)$ is connected, the $x$-component $H_x$ of $f^{-1}T$ is contained in an interior of an $n$-cell $C$ in $M$ and $f|_{H_x} \colon H_x \to T$ is a normal branched covering. Note that, e.g.\;by the Alexander duality, each boundary component of $\partial H_x$ separates $C$ into exactly two connected components. Since $M\setminus H_x \supset f^{-1}(N\setminus T)$ is connected, we conclude that the boundary $\partial H_x$ is connected. On the other hand, since $f|_{\partial H_x} \colon \partial H_x \to \partial T$ is a covering, we have that $\partial H_x$ is a $2$-torus.
Thus the mapping $f|_{H_x}\colon H_x\to T$ satisfies the conditions of Proposition \ref{prop:CH_cor}.
\end{proof}


\def\cprime{$'$}

\end{document}